\documentclass[12pt]{amsart}
\usepackage{amssymb}
\usepackage{mathrsfs}
\usepackage{amssymb}
\usepackage{amsfonts}
\usepackage[english]{babel}
\usepackage[T1]{fontenc}
\usepackage[latin1]{inputenc}
\usepackage{fullpage}
\usepackage{amsmath}
\usepackage[all]{xy}
\usepackage{stmaryrd}

\newtheorem{theorem}{Theorem}[section]
\newtheorem{lemma}[theorem]{Lemma}
\newtheorem{corollary}[theorem]{Corollary}
\theoremstyle{definition}
\newtheorem{definition}[theorem]{Definition}

\newtheorem{proposition}[theorem]{Proposition}
\newtheorem{conjecture}[theorem]{Conjecture}

\theoremstyle{remark}
\newtheorem{remark}[theorem]{Remark}

\numberwithin{equation}{section}

\makeatletter
\let\origsection\section
\renewcommand\section{\@ifstar{\starsection}{\nostarsection}}

\newcommand\nostarsection[1]
{\sectionprelude\origsection{#1}\sectionpostlude}

\newcommand\starsection[1]
{\sectionprelude\origsection*{#1}\sectionpostlude}

\newcommand\sectionprelude{%
  \vspace{1em}
}

\newcommand\sectionpostlude{%
  \vspace{1em}
}
\makeatother

\begin{document}

\title{On Brown-York mass and compactly conformal deformations of scalar curvature}



\author{Wei Yuan}
\address{(Wei Yuan) Department of Mathematics, University of California, Santa Cruz, CA 95064}
\email{wyuan2@ucsc.edu}




\thanks{This work is support by NSF grant DMS-1005295 and DMS-1303543. }

\begin{abstract}
In this article, we found a connection between Brown-York mass and the first Dirichlet Eigenvalue of a Schr\"odingier type operator. In particular, we proved a local positive mass type theorem for metrics conformal to the background one with suitable presumptions. As applications, we investigated compactly conformal deformations which either increase or decrease scalar curvature. We found local conformal rigidity phenomena occur in both cases for small domains and as for manifolds with nonpositive scalar curvature it is even more rigid in particular. On the other hand, such deformations exist for closed manifolds with positive scalar curvature. We also constructed such kind of deformations on a type of product manifolds that either increase or decrease their scalar curvature compactly and conformally.  These results together answered a natural question arises in \cite{Corvino, Lohkamp}.
\end{abstract}

\maketitle

\section{Introduction}

One of the most important results in differential geometry during the passed decades is the \emph{Positive Mass Theorem} proved by Schoen and Yau (\cite{S-Y1, S-Y2}) and Witten (\cite{Witten}) in 1980s. It has been a source of inspiration of many interesting works since it appeared. Among these works, rigidity concerning deformations with nondecreasing scalar curvature is one of the central topics. \\

In fact, before \emph{Positive Mass Theorem} was proved, Fischer and Marsden had proved a rigidity result for perturbations of flat metrics on closed manifolds in 1975 (\cite{F-M}).  As for domains, it was observed by Miao that \emph{Positive Mass Theorem} implies the rigidity of unit Euclidean ball with respect to deformations of nondecreasing scalar curvature inside and mean curvature on the boundary, provided the induced metric on the boundary is spherical. Around the same time, Shi and Tam showed that this rigidity holds for any convex domain in the Euclidean space. In fact, they proved that for such a domain its Brown-York mass is nonnegative and vanishes if and only it is Euclidean (\cite{S-T}). \\

An alternative way to generalize the rigidity part of \emph{Positive Mass Theorem} was done by Min-Oo. He proved a similar rigidity result for strongly asymptotically hyperbolic spin manifold in 1989 (\cite{Min-Oo1}) . As a quick corollary, it implies that any compact deformation of hyperbolic metric, which keeps the scalar curvature nondecreasing has to be isometric to the canonical hyperbolic metric. In fact, the strong asymptotically hyperbolic presumption in Min-Oo's result can be reduced to a weaker one due to Andersson and Dahl's work (\cite{A-D}).\\

Motivated by the rigidity part of \emph{Positive Mass Theorem} and its analogue in the asymptotically hyperbolic setting, Min-Oo proposed the following conjecture in 1995.

\begin{conjecture} [Min-Oo \cite{Min-Oo2}]
For $n\geq 2$, suppose $g$ is a smooth metric on the upper hemisphere $S^n_{+}$, which satisfies the following properties,
\begin{itemize}
\item $R(g) \geq n(n-1)$
\item $g|_{_{\partial S^n_+}} = g_{_{S^{n-1}}}$
\item $\partial S^n_+$ is totally geodesic with respect to $g$.
\end{itemize}
Then $(S^n_+, g)$ is isometric to the canonical upper hemisphere.
\end{conjecture}

In dimension 2, the classic theorem of Toponogov shows Min-Oo's conjecture is actually true. As for higher dimensions, when restricted in the conformal class of the canonical spherical metric, Hang and Wang gave an affirmative answer to the conjecture in 2006 (\cite{H-W1}). In fact, their result is sharp in the sense that if the support of the deformation contains the upper hemisphere strictly, then there is a conformal deformation satisfies the boundary condition with nondecreasing scalar curvature but strictly increasing somewhere.  \\

Although people made many attempts trying to solve Min-Oo's conjecture, unfortunately it was disproved in 2010. Brendle, Marques and Neves constructed a counter example for Min-Oo's Conjecture by combining technics of perturbation and gluing (\cite{B-M-N}, see also \cite{C-L-M} for a generalization). \\

However, when considering a geodesic ball strictly contained in upper hemisphere, Brendle and Marques showed that rigidity phenomena do occur if the size of the geodesic ball is less than a certain number and the metric is not far away from the canonical spherical one (\cite{B-M}). (For an improvement of the size of the geodesic ball, see \cite{C-M-T}). We refer the excellent survey article \cite{Brendle} for those who are interested in the history of Min-Oo conjecture. \\

Inspired by \cite{B-M}, Qing and the author generalized above rigidity result to generic vacuum static spaces ({\cite{Q-Y2}}). This is in fact a sharp rigidity result due to Corvino's work on the stability of non-vacuum static domains (\cite{Corvino}). For a detailed discussion of vacuum static spaces, please see \cite{Q-Y1}.\\

When restricted in conformal deformations, Qing and the author achieved a sharp rigidity result for vacuum static spaces with positive scalar curvature  (c.f. \cite{Q-Y2}), which generalized Hang and Wang's work in \cite{H-W1} on upper hemisphere. On the other hand, motivated by a question proposed by Escobar (\cite{Escobar}), Barbosa, Mirandola and Vitorio found an elegant integral identity and with the aid of which they proved the rigidity part independently in a more general setting (\cite{B-M-V}). )\\

On the other hand, as a comparison, it would be interesting to know what happens if we require deformations decrease scalar curvature instead of increase it as what we discussed previously. In fact, not many works were known with respect to this question to the best of the authors knowledge. Among them, Lohkamp's result is the most well-known one (see \cite{Lohkamp}). He showed that there is a generic compact deformation which decrease scalar curvature for an arbitrary Riemannian manifold. However, it was not clear that such a deformation can be realized within conformal classes or not due to Lohkamp's proof. So we would like to ask the question: does such a conformal deformation exists? Moreover, we can ask similar questions for Corvino's constructions (see \cite{Corvino}).\\

Before we answer these questions, we introduce the following well-known notion due to Brown and York (see \cite{B-Y1, B-Y2}). For the purpose of our article, we restrict metrics in the conformal class of the background metric.

\begin{definition}
For $n \geq 2$, let $(\Omega^n, \bar g)$ be an $n$-dimensional compact Riemannian manifold with smooth boundary $\partial \Omega$. Then for any metric $g \in [\bar g]$ on $\Omega$ with $g = \bar g$ on $\partial \Omega$, the \emph{Brown-York mass relative to $\bar g$} is defined to be the quantity 
$$
m_{BY}( \partial \Omega, \bar g; g ) = \int_{\partial \Omega} \left( H_{\bar g} - H_g \right)\ d\sigma_{\bar g},
$$ where $H_{\bar g}$ and $H_g$ are mean curvatures of $\bar g$ and $g$ of $\partial \Omega$ with respect to the outward normals respectively.
\end{definition}

Another notion we will used frequently is the following one:
\begin{definition}
We denote the first Dirichlet eigenvalue of the Schr\"odinger type operator $$\mathscr{L}_{\bar g} : = - \Delta_{\bar g} - \frac{R_{\bar g}}{n-1}$$ on the domain $\Omega \subset M$ as $$\Lambda_1(\Omega, \bar g) = \inf_{\varphi \in C_0^\infty (\Omega)} \frac{\int_\Omega\ \varphi\ \mathscr{L}_{\bar g} \varphi \ dv_{\bar g}}{ \int_\Omega \varphi^2\ dv_{\bar g}}$$
and $\phi \not\equiv 0$ is a corresponding eigenfunction, if $$\mathscr{L}_{\bar g} \phi = \Lambda_1(\Omega, \bar g) \phi.$$
\end{definition}

Note that the operator $\mathscr{L}_{\bar g}$ appeared as the trace of $L^2$-formal adjoint of the linearization of scalar curvature up to a constant (see \cite{F-M}): $$\gamma_{\bar g}^* = Hess_{\bar g} - \bar g\Delta_{\bar g} - Ric_{\bar g} : C^\infty(M) \rightarrow S_2(M).$$

Now we can state our main theorem as follow, which can be viewed as an analogue of Shi and Tam's positive mass theorem (cf. \cite{S-T}):

\begin{theorem}\label{thm:rigidity}
For $n \geq 2$, let $(M^n,\bar{g})$ be an $n$-dimensional compact Riemannian manifold with smooth boundary. Suppose the first Dirichlet eigenvalue of $\mathscr{L}_{\bar g}$ on $M$ satisfies that $$\Lambda_1 (M, \bar g) > 0.$$

$(I)$ :  For any metric $g_+ \in [\bar{g}]$ with 
$$ g_+ = \bar g$$ 
on $\partial M$ and in addition, 
$$
 R_{g_+} \geq R_{\bar{g}} \quad \text{on  $M$},
$$
we have 
$$
m_{BY}( \partial M, \bar g; g_+ ) \geq 0
$$
and equality holds if and only if when $g_+ = \bar g$ on $M$.\\

$(II)$: For any $\delta \in (0,1)$ and any metric $g_- \in [\bar{g}]$ with $$||g_-||_{C^0 (M, \bar g)} < \alpha := \left\{ \aligned  &\sqrt{n} \left(  1 + \frac{(n-1) \delta \Lambda_1(M, \bar g)}{\max_M R_{\bar g}} \right)^{\frac{4}{n+2}}   &\quad \text{ when $\max_{M} R_{\bar g} > 0$}\\ 
& + \infty,  &\quad \text{ when $\max_{M} R_{\bar g} \leq 0$}
\endaligned\right.$$ and 
$$g_- = \bar g$$ 
on $\partial M$ also in addition, 
$$
R_{g_-} \leq R_{\bar{g}}  \quad \text{on  $M$}
$$
we have 
$$
m_{BY}( \partial M, \bar g; g_- ) \leq 0
$$
and equality holds if and only if when $g_- = \bar g$ on $M$.
\end{theorem}

\begin{remark}
When $\bar g$ is flat and $g$ is asymptotically flat, the presumption $$R_g \geq R_{\bar g} = 0$$ is usually referred as \emph{Dominant Energy Condition} in general relativity.
\end{remark}

If the diameter of the manifold is sufficiently small, the first eigenvalue of Laplacian can be sufficiently large. Hence the above theorem always holds on a sufficiently small domain in an arbitrary Riemannian manifold:

\begin{corollary}\label{cor:rig_small_domain}
For $n \geq 2$, let $(M^n,\bar{g})$ be an $n$-dimensional Riemannian manifold.  For any $p \in M$, there exists an $r_0 > 0$, such that for any domain $\Omega\subset M$ contains $p$ with $\Omega \subset B_{r_0}(p)$, the metric $g \in [\bar{g}]$ is a conformal deformation of $\bar g$ supported in $\Omega$ which satisfies either
$$R_g \geq R_{\bar{g}}$$ or $$R_g \leq R_{\bar g}$$ with $||g - \bar g||_{C^0(M, \bar g)}$ sufficiently small on $M$. Then we have $g \equiv \bar{g}$ on $M$. In particular, this implies that $\bar g$ is an isolated solution of the problem
$$
\left\{  \aligned &R(g) = R_{\bar g}\\ 
&g\in [\bar g] \\ 
&supp(g - \bar g) \subset \Omega,
\endaligned\right.
$$
if the diameter of $\Omega$ is sufficiently small.

\end{corollary}

This corollary shows that in particular there is no conformal perturbation increase or decrease scalar curvature with sufficiently small support. It suggests that deformations result in \cite{ Corvino, Lohkamp} is sharp in the sense that the generic deformations out of conformal classes are necessary.\\

Another interesting conclusion is that Brown-York mass behave perfect on manifolds with nonpositive scalar curvature, which suggests that they are much more rigid in terms of compactly conformal deformations.

\begin{corollary}\label{cor:non_pos_scalar}
For $n \geq 2$, let $(M, \bar g)$ be a Riemannian manifold with scalar curvature $R_{\bar g} \leq 0$. Then for any compactly contained domain $\Omega \varsubsetneqq M$ and any metric $g \in [\bar g]$ on $\Omega$ with $$ g = \bar g$$ on $\partial \Omega$.

$(I)$: Suppose
$$
R_g \geq R_{\bar{g}}  \quad \text{on  $\Omega$},
$$
then $$
m_{BY}( \partial \Omega, \bar g; g ) \geq 0.
$$

$(II)$: Suppose
$$
R_g \leq R_{\bar{g}}  \quad \text{on  $\Omega$},
$$
then $$
m_{BY}( \partial \Omega, \bar g; g ) \leq 0.
$$
In either case, the equality holds if and only $ g = \bar g$ on $\Omega$.

In particular, these imply that $g = \bar g$ is the unique solution to the following problem 
$$
\left\{  \aligned &R(g) = R_{\bar g}\\ 
&g\in [\bar g] \\ 
&supp(g - \bar g) \subset \Omega.
\endaligned\right.
$$

\end{corollary}

It is natural to ask what happens to manifolds with positive scalar curvature? We have already known that if the domain is sufficiently small, rigidity phenomena occur. So it is interesting to investigate the phenomena in large domains. In fact, we have the following deformation result, which suggests that the vanishing of Brown-York mass won't imply the conformal rigidity as it does previously:

\begin{theorem}\label{thm:deformation}
For $n \geq 2$, let $(M, \bar g)$ be a Riemannian manifold with scalar curvature $R_{\bar g} >0$. Suppose there exists a compactly contained domain $\Omega \subset M$ with the first Dirichlet eigenvalue of $\mathscr{L}_{\bar g}$ satisfies that
$$
\Lambda_1(\Omega, \bar g) < 0.
$$

Then there are smooth metrics $g_+, g_- \in [\bar g]$ such that $$supp\ (\bar g - g_+),\ supp\ (\bar g - g_-) \subset \Omega $$ and scalar curvatures satisfy that $$R_{g_+} \geq R_{\bar g}$$ and  $$R_{g_-} \leq R_{\bar g}$$ with strict inequality holding inside $\Omega$ respectively. 
\end{theorem}

If the manifold is closed, such domains always exists. Hence we can always have the following deformation result:
\begin{corollary}\label{cor:non_rig_closed}
For $n \geq 2$, let $(M, \bar g)$ be a closed Riemannian manifold with scalar curvature $R_{\bar g} >0$. Then there is a compactly contained domain $\Omega \subset M$ and smooth metrics $g_+, g_- \in [\bar g]$ such that $$supp\ (\bar g - g_+),\ supp\ (\bar g - g_-) \subset \Omega $$ and scalar curvatures satisfy that $$R_{g_+} \geq R_{\bar g}$$ and  $$R_{g_-} \leq R_{\bar g}$$ respectively with strict inequality holding inside $\Omega$. 

\end{corollary}

However, if the manifold $(M, \bar g)$ is complete but noncompact, whether the positivity of scalar curvatures implies the existence of such domains is not clear. But if $(M, \bar g)$ has quadratic volume growth, such domains do exists and thus we have the following deformation result: 

\begin{corollary}\label{cor:non_rig_noncpt}
For $n \geq 2$, let $(M, \bar g)$ be a complete noncompact Riemannian manifold with scalar curvature $R_{\bar g} > Q > 0$, where $Q$ is a positive constant. Suppose that $(M, \bar g)$ has quadratic volume growth, then there is a compactly contained domain $\Omega \subset M$ and smooth metrics $g_+, g_- \in [\bar g]$ such that $$supp\ (\bar g - g_+),\ supp\ (\bar g - g_-) \subset \Omega $$ and scalar curvatures satisfy that $$R_{g_+} \geq R_{\bar g}$$ and  $$R_{g_-} \leq R_{\bar g}$$ respectively with strict inequality holding inside $\Omega$. 

\end{corollary}

Finally, we would like to discuss the subtle issue of domains which has critical eigenvalue $$\Lambda_1(\Omega, \bar g) = 0$$ briefly. According to \cite{B-M-V}, in order to get rigidity we only need to assume $g = \bar g$ on $\partial \Omega$ and $R_g \geq R_{\bar g}$ on $\Omega$, which means the Brown-York mass can only be vanished in this case and we won't get its positivity. But it is not clear right now whether the same phenomena appear when we assume scalar curvature nonincreasing instead and we hope to address this issue in the future.\\

The article is organized as follow: Section 2 is mainly about the discussion of Brown-York mass and conformal rigidity where Theorem \ref{thm:rigidity}, Corollary \ref{cor:rig_small_domain} and \ref{cor:non_pos_scalar} will be proved; Section 3 is devoted to the construction of conformal deformations which suggest the rigidity fails.\\

\paragraph{\textbf{Acknowledgement}}
The author would like to express his deepest appreciations to Professor Jie Qing for his inspiring discussions and constantly encouragements. \\

\section{Brown-York mass and rigidity of compactly conformal deformations}

In this section, we will prove results associated to the Brown-York mass and investigate rigidity phenomena of conformally compact deformations.

\begin{proof}[Proof of Theorem \ref{thm:rigidity}]

For $g\in [\bar g]$, there exists a smooth function $u$ such that
$$
g = \left\{ \aligned  e^{2u}\bar g  & \quad \text{ when $n=2$}\\ 
u^\frac 4{n-2}\bar g & \quad \text{ when $n\geq 3$}.
\endaligned\right.
$$

It is well known that for $u = 1$ on $\partial M$,
$$
R_g = \left\{\aligned e^{-2u} \left( R_{\bar{g}} - 2 \Delta_{\bar g} u \right)  & \quad \text{ when $n=2$}\\ 
 u^{-\frac{n+2}{n-2}} \left( R_{\bar{g}} u - \frac {4(n-1)}{n-2}\Delta_{\bar g} u \right) & \quad \text{ when $n\geq 3$} 
 \endaligned\right.
$$
and
$$  
 H_g = \left\{ \aligned H_{\bar g} + 2\partial_\nu u   & \quad \text{ when $n=2$}\\ 
 H_{\bar g} + \frac {2(n-1)}{n-2}\partial_\nu u & \quad \text{ when $n\geq 3$},
\endaligned\right.
 $$
where $\nu$ is the outward normal of $\partial M$ with respect to $\bar g$.\\
 
Let 
\begin{align*}
w (x) =
\begin{cases}
\frac{ R_{\bar{g}} (e^{2u(x)} - 1)}{2 u(x)} = R_{\bar g} e^{2\xi} & \quad \text{ when $n = 2$} \\
\frac{\frac {n-2}{4(n-1)} R_{\bar{g}} u(x) \left(u(x)^{\frac{4}{n-2}} - 1\right)}{u(x)-1} = \frac {R_{\bar g}}{n-1} \xi^\frac 4{n-2} \frac u \xi 
& \quad \text{ when $n\geq 3$},
\end{cases}
\end{align*}
where $\xi(x)$ is between $0$ and $u(x)$ when $n=2$; $\xi(x)$ is between $1$ and $u(x)$ when $n\geq 3$.\\

Assume that for $(I)$, 
$$
m_{BY}(\partial M, \bar g; g_+ ) \leq 0;
$$
and for $(II)$,
$$
m_{BY}( \partial M, \bar g; g_- ) \geq 0.
$$
We are going to show that in both cases $g_+ = g_- = \bar g$ on $M$ and thus the theorem follows.\\

Let $u_+$ be the smooth function for the conformal factor of $g_+$ and $u_-$ for $g_-$ respectively.\\

We take for $(I)$, 
\begin{align*}
v(x) = \begin{cases} u_+(x) & \quad \text{ when $n = 2$} \\
u_+(x) - 1& \quad \text{ when $n\geq 3$}
\end{cases}
\end{align*}
and for $(II)$
\begin{align*}
v(x) = 
\begin{cases} 
-u_-(x) & \quad \text{ when $n = 2$} \\
1 - u_-(x) & \quad \text{ when $n\geq 3$}.
\end{cases}
\end{align*}

Thus we have 
\begin{equation*}\label{scalar-cur}
\left\{\aligned - \Delta_{\bar g} v - w(x) v & \geq 0 \quad \text{ in $M$}\\
v & = 0 \quad \text{ on $\partial M$}\\
\int_{\partial M} \partial_{\nu} v\ d\sigma_{\bar g} &\geq 0.\\
\endaligned\right.
\end{equation*}

We claim that $$v (x) \geq 0$$ on $M$.\\

Suppose not, there exists a point $\bar{x} \in M$ such that $$v(\bar{x}) < 0.$$

For any $\varepsilon > 0$, let $ M_\varepsilon$ be an $\varepsilon$-extension of $(M, \bar g)$ along the boundary $\partial M$. That is, we extend $\bar g$ smoothly on a smooth manifold $M_\varepsilon$ with boundary $\partial M_\varepsilon$ which contains the interior of $M$ as an open submanifold and satisfies that $dist_{\bar g}(\partial M, \partial M_{\varepsilon}) = \varepsilon$.\\

Since $\Lambda_1 (M, \bar g) > 0$, for any $0< \delta < 1$, there exists an $\varepsilon > 0$ small such that $$\Lambda_1 (M_\varepsilon, \bar g) > \delta \Lambda_1 (M, \bar g) > 0.$$

Let $\phi$ be a first eigenfunction associated to $\Lambda_1 (M_\varepsilon, \bar g)$ and from the standard theory we can choose $\phi$ to be positive on $M_\varepsilon$. Take $\varphi (x) = \frac{v(x)}{\phi(x)}$, then $$\varphi \equiv  0$$ on $\partial M$ and $$\varphi(\bar{x}) < 0.$$

Let $x_0 \in M$ be the point where $\varphi$ attains its minimum on $M$. Then we have $$- \Delta_{\bar{g}} \varphi (x_0) \leq 0, \ \ \ \nabla \varphi(x_0) = 0.$$

Clearly, $\varphi(x_0) < 0$, which would imply $v(x_0) < 0$ and thus $w(x_0) < \frac{R_{\bar g}}{n-1} < \frac{R_{\bar g}}{n-1} + \delta \Lambda_1 (M, \bar g)$. This is trivial for $(I)$ and for $(II)$ with $\max_M R_{\bar g} \leq 0$; as for the case of $\max_M R_{\bar g} > 0$, it can be deduced from the presumption that $$||g_-||_{C^0 (M, \bar g)} < \alpha = \sqrt{n} \left(  1 + \frac{(n-1) \delta \Lambda_1(M, \bar g)}{\max_M R_{\bar g}} \right)^{\frac{4}{n+2}} .$$ \\

In fact, for $n \geq 3$, $v(x_0) < 0$ implies that $1 < \xi(x_0) < u_-(x_0)$ and thus $$w(x_0) \leq \frac{R_{\bar g}}{n-1} u_-^{\frac{n+2}{n-2}}(x_0) \leq \frac{R_{\bar g}}{n-1} \left(\max_{M} u_- \right)^{\frac{n+2}{n-2}} = \frac{R_{\bar g}}{n-1}    \left(\frac{||g_-||_{C^0 (M, \bar g )}}{\sqrt{n}}\right)^{\frac{n+2}{4}}   \leq \frac{R_{\bar g}}{n-1} + \delta \Lambda_1 (M, \bar g),$$ where $||g_-||_{C^0 (M, \bar g)} = \sqrt{n}  \left(\max_{M} u_- \right)^{\frac{4}{n-2}}$. \\

The estimate for $n =2$ follows with similar calculations.\\

Now at $x_0$, we have
\begin{align*}
0 &\geq - \Delta_{\bar{g}} \varphi (x_0)\\
&= \frac{1}{\phi(x_0)} \left( - \Delta_{\bar{g}} v (x_0) + \frac{v(x_0)}{\phi(x_0)} \Delta_{\bar{g}} \phi(x_0) \right) + \frac{2}{\phi(x_0)} \nabla \varphi(x_0) \cdot \nabla \phi(x_0)\\
&= \frac{1}{\phi(x_0)} \left( - \Delta_{\bar{g}} v (x_0) - \left(\frac{R_{\bar g}}{n-1} + \Lambda_1(M_{\varepsilon}, \bar g) \right) v(x_0) \right)\\
&\geq \frac{1}{\phi(x_0)} \left( w(x_0) - \frac{R_{\bar g}}{n-1} - \Lambda_1 (M_\varepsilon, \bar g) \right)v (x_0)\\
&\geq  - \frac{1}{\phi(x_0)}\ \left( \Lambda_1(M_\varepsilon, \bar g) - \delta \Lambda_1(M, \bar g)  \right)\ v (x_0)\\
&> 0.
\end{align*}

Contradiction! Therefore, $$v(x) \geq 0, \ \ \ \forall x \in M.$$

On the other hand, the non-negativity of $v$ implies that  $$\partial_{\nu} v (x) < 0$$ for any $x \in \partial M$ by Generalized Hopf's lemma (cf. Theorem 7.3.3 in \cite{C-L}), unless $v$ is identically vanishing. Thus 
$$
\int_{\partial M} \partial_{\nu} v\ d\sigma_{\bar g} < 0,
$$
which contradicts to the fact
$$
\int_{\partial M} \partial_{\nu} v\ d\sigma_{\bar g} \geq 0,
$$
 hence $v \equiv 0$ on $M$.
\emph{i.e.} $$g \equiv \bar{g}.$$

\end{proof}

In order to prove Corollary \ref{cor:rig_small_domain}, we recall the following well-known fact on the estimate of the first eigenvalue of Laplacian.

\begin{lemma}(Karp-Pinsky, \cite{K-P}) \label{eigenvalue_rescaling}
For any $p \in M$, there exist positive constants $r_0$ and $c_0$ such that the first eigenvalue of Laplacian satisfies that
\begin{align*}
\lambda_1(B_r(p), \bar{g}) \geq \frac{c_0}{r^2}, 
\end{align*}
provided $r < r_0$.
\end{lemma}

\begin{proof}[Proof of corollary \ref{cor:rig_small_domain}]
Take $$\Lambda:= \max_{x \in B_1(p)} \frac{R_{\bar g}(x)}{n-1}.$$ We can choose $0< r_0 < 1$ such that 
\begin{align*}
\lambda_1(B_{r_0}(p), \bar{g}) \geq \frac{c_0}{r_0^2} > \Lambda, 
\end{align*}
due to Lemma \ref{eigenvalue_rescaling}. Clearly, 
\begin{align*}
\Lambda_1(B_{r_0}(p), \bar g) \geq \ \lambda_1(B_{r_0}(p), \bar{g}) - \Lambda>\ 0.
\end{align*}

Now the conclusion follows by applying Theorem \ref{thm:rigidity} on the geodesic ball $B_{r_0}(p)$.
\end{proof}

We finish this section by showing Corollary \ref{cor:non_pos_scalar} is true.
\begin{proof}[Proof of corollary \ref{cor:non_pos_scalar}]
Note that if $R_{\bar g} \leq 0$, then 
$$\Lambda_1(\Omega, \bar g) > 0$$ for any compactly contained domain $\Omega \subset M$.  The conclusion follows automatically from Theorem \ref{thm:rigidity}.
\end{proof}

\section{Non-rigidity phenomena of compactly conformal deformations}

In this section, we will construct compactly conformal deformations which suggests that the vanishing of Brown-York mass does not imply the conformal rigidity for large domains, assuming the positivity of scalar curvature.\\

Let $\phi$ be a first Dirichlet eigenfunction of Laplacian on $\Omega$. In particular, we can choose $\phi > 0$ on $\Omega$. Note that, in fact $\partial_\nu \phi < 0$ on $\partial \Omega$ by Generalized Hopf's lemma (cf. Theorem 7.3.3 in \cite{C-L}). Following the idea of Lemma 20 and Proposition 21 in \cite{B-M}, we can construct metrics such that they satisfy the desired properties near $\partial \Omega$:

\begin{proposition}\label{prop:boundary_metric}
For any $\varepsilon > 0$ sufficiently small, there are smooth metrics $\tilde g_+$ and $\tilde g_-$ on $M$ such that 
\begin{align*}
(I):
\begin{cases}
R_{\tilde g_+} > R_{\bar g}  &\text{on $\Omega - \Omega_{2\varepsilon}$}\\
\tilde g_+ = \bar g &\text{on  $M - \Omega$}
\end{cases}
\end{align*}
and
\begin{align*}
(II):
\begin{cases}
R_{\tilde g_-} < R_{\bar g}  &\text{on $\Omega - \Omega_{2\varepsilon}$}\\
\tilde g_- = \bar g &\text{on $M - \Omega$},
\end{cases}
\end{align*}
where $\Omega_{2\varepsilon} := \{ x \in \Omega: \phi(x) > 2\varepsilon\}$.
\end{proposition}

\begin{proof}
Since $\partial_\nu \phi \neq 0$ on $\partial \Omega$, we can find an $\varepsilon > 0$ such that 
\begin{align*}
\Delta_{\bar g} e^{-\frac{1}{\phi}} = e^{-\frac{1}{\phi}} \left( \frac{|\nabla \phi|^2 -\lambda_1(\Omega, \bar g)\phi ^3- 2 |\nabla \phi|^2 \phi }{\phi ^4}\right) \geq 0,
\end{align*}
\emph{i.e.} $e^{-\frac{1}{\phi }}$ is subharmonic on $\Omega - \Omega_{2\varepsilon}$. \\

For $n \geq 3$, we take $w_+ = 1 - e^{-\frac{1}{\phi }}$, $w_- = 1 + e^{ - \frac{1}{\phi }}$ on $\Omega$ and extend them to be constantly $1$ outside domain $\Omega$. Clearly, $w_+$ and $w_-$ are smooth on $M$. Now let $\tilde g_+ = w_+^{\frac{4}{n-2}} \bar g$ and $\tilde g_- = w_-^{\frac{4}{n-2}} \bar g$ respectively. Then we have 
\begin{align*}
R_{\tilde g_+} = w_+^{- \frac{n+2}{n-2}} \left( R_{\bar g}w_+ - \frac{4(n-1)}{n-2} \Delta_{\bar g} w_+\right) \geq w_+^{- \frac{4}{n-2}} R_{\bar g} > R_{\bar g}
\end{align*}
and 
\begin{align*}
R_{\tilde g_-} = w_-^{- \frac{n+2}{n-2}} \left( R_{\bar g}w_- - \frac{4(n-1)}{n-2} \Delta_{\bar g} w_-\right) \leq w_-^{- \frac{4}{n-2}} R_{\bar g} < R_{\bar g}
\end{align*}
on $\Omega - \Omega_{2\varepsilon}$.\\

Similar constructions work for $n = 2$, if we take $\tilde g_+ = w_+^2 \bar g$, $\tilde g_- = w_-^2 \bar g$ respectively.
\end{proof}

Now let $$\Omega_{\varepsilon} := \{ x \in \Omega: \phi(x) > \varepsilon\}$$ and choose $\varepsilon$ sufficiently small such that $$\Lambda_1(\Omega_\varepsilon, \bar g) < 0$$ and let $\psi$ be a positive first eigenfunction of $\mathscr{L}_{\bar g}$ on $\Omega_\varepsilon$ which vanishes on $\partial \Omega_\varepsilon$. We will produce perturbed metrics as follow:

\begin{proposition}
There are metrics $g^+_t, g^-_t \in [\bar g]$ with 
\begin{align*}
(I):
\begin{cases}
R_{g^+_t} \geq R_{\bar g}  &\text{on $\bar \Omega_{\varepsilon }$}\\
g^+_t = \bar g &\text{on $\partial \Omega_\varepsilon$}\\
H_{g^+_t} > H_{\bar g} &\text{on $\partial \Omega_\varepsilon$}
\end{cases}
\end{align*}
and
\begin{align*}
(II):
\begin{cases}
R_{g^-_t} \leq R_{\bar g}  &\text{on $\bar \Omega_{\varepsilon }$}\\
g^-_t = \bar g &\text{on $\partial \Omega_\varepsilon$}\\
H_{g^-_t} < H_{\bar g} &\text{on $\partial \Omega_\varepsilon$}.
\end{cases}
\end{align*}
\end{proposition} 
 
\begin{proof} 
 For $n \geq 3$, let $g_t= u_t^{\frac{4}{n-2}} \bar g$, where $u_t = 1 + t  \varphi $, where $\varphi \in C^\infty (M)$ and supported in $\Omega$. Then 
\begin{align*}
R_{g_t} = 
& u_t^{-\frac{n+2}{n-2}} \left( R_{\bar{g}} u_t - \frac {4(n-1)}{n-2}\Delta_{\bar g} u_t \right)\\
=& R_{\bar g} - \frac{4(n-1)}{n-2}\left( \Delta_{\bar g} \varphi + \frac{R_{\bar g}}{n-1} \varphi \right)t + O(t^2)
\end{align*}
and 
\begin{align*}
 H_{g_t} = H_{\bar g} + t \frac {2(n-1)}{n-2}\partial_\nu \varphi.
\end{align*}

We take $$\varphi_+ := - \psi$$ for $(I)$ and $$\varphi_- := \psi $$ for $(II)$.\\

Then we have
$$\Delta_{\bar g} \varphi_+ + \frac{R_{\bar g}}{n-1} \varphi_+ = \mathscr{L}_{\bar g} \psi = \Lambda_1(\Omega_\varepsilon, \bar g) \psi  < 0$$ on $\Omega$ and $\varphi_+ = 0$, $\partial_\nu\varphi_+ > 0$ on $\partial \Omega_\varepsilon$.\\

Similarly, $$\Delta_{\bar g} \varphi_- + \frac{R_{\bar g}}{n-1} \varphi_- > 0$$ on $\Omega$ and $\varphi_- = 0$, $\partial_\nu\varphi_- < 0$ on $\partial \Omega_\varepsilon$.\\

Let $$g_t^+ = (u^+_t)^{\frac{4}{n-2}} \bar g$$ and $$g_t^- = (u^-_t)^{\frac{4}{n-2}} \bar g,$$ where $u^+_t = 1 + t  \varphi_+ $ and $u^-_t = 1 + t  \varphi_- $ respectively. Then they are our desired metrics, if we choose $t$ sufficiently small.\\

For $n=2$, we take $g_t^+ = (u^+_t)^2 \bar g$ and $g_t^- = (u^-_t)^2 \bar g$. By similar calculation, we can see they satisfy $(I)$ and $(II)$ respectively.
\end{proof}

In order to glue metrics we derived previously, we need to match them at zero's order first.

\begin{proposition}\label{prop:perturb_metric}
There are metrics $\hat g^+_t, \hat g^-_t \in [\bar g]$ with 
\begin{align*}
(I'):
\begin{cases}
R_{\hat g^+_t} > R_{\bar g}  &\text{on $\bar \Omega_{\varepsilon }$}\\
\hat g^+_t = \tilde g_+ &\text{on $\partial \Omega_\varepsilon$}\\
H_{\hat g^+_t} > H_{\bar g} &\text{on $\partial \Omega_\varepsilon$}
\end{cases}
\end{align*}
and
\begin{align*}
(II'):
\begin{cases}
R_{\hat g^-_t} < R_{\bar g}  &\text{on $\bar \Omega_{\varepsilon }$}\\
\hat g^-_t = \tilde g_- &\text{on $\partial \Omega_\varepsilon$}\\
H_{\hat g^-_t} < H_{\bar g} &\text{on $\partial \Omega_\varepsilon$}.
\end{cases}
\end{align*}

\end{proposition}

\begin{proof}
Let $\hat g^+_t := (1 - e^{-\frac{1}{\varepsilon}})^{\frac{4}{n-2}} g^+_t$ and $\hat g^-_t := (1 + e^{-\frac{1}{\varepsilon}})^{\frac{4}{n-2}} g^-_t$. Then clearly, $\hat g^+_t = \tilde g_+$ and $\hat g^-_t = \tilde g_-$ on $\partial \Omega_\varepsilon$ with 
$$
R_{\hat g^+_t} = (1 - e^{-\frac{1}{\varepsilon}})^{ - \frac{4}{n-2}} R_{g^+_t} > R_{g^+_t} \geq R_{\bar g}
$$
and 
$$
R_{\hat g^-_t} = (1 + e^{-\frac{1}{\varepsilon}})^{ - \frac{4}{n-2}} R_{g^-_t} < R_{g^-_t} \leq R_{\bar g}
$$
on $\bar \Omega_\varepsilon$.

As for mean curvatures,
\begin{align*}
H_{\hat g^+_t} &= (1 - e^{-\frac{1}{\varepsilon}})^{ - \frac{2}{n-2}} \left( H_{\bar g} + \frac{2(n-1)}{n-2}  t (- \partial_\nu \psi) \right) \\
&> (1 - e^{-\frac{1}{\varepsilon}})^{ - \frac{2}{n-2}} \left( H_{\bar g} + \frac{2(n-1)}{n-2} \cdot \frac{ e^{-\frac{1}{\varepsilon}}}{\varepsilon^2 (1 - e^{-\frac{1}{\varepsilon}})} (- \partial_\nu \phi) \right)\\
&= H_{\tilde g_+}
\end{align*}
and
\begin{align*}
H_{\hat g^-_t} &=  (1 + e^{-\frac{1}{\varepsilon}})^{ - \frac{2}{n-2}} \left( H_{\bar g} - \frac{2(n-1)}{n-2}  t (- \partial_\nu \psi) \right) \\
&< (1 + e^{-\frac{1}{\varepsilon}})^{ - \frac{2}{n-2}} \left( H_{\bar g} - \frac{2(n-1)}{n-2} \cdot \frac{ e^{-\frac{1}{\varepsilon}}}{\varepsilon^2 (1 + e^{-\frac{1}{\varepsilon}})} (- \partial_\nu \phi) \right)\\
&= H_{\tilde g_-}
\end{align*}
on $\partial \Omega_\varepsilon$, if we choose $\varepsilon > 0$ sufficiently small.

\end{proof}

The following crucial gluing theorem (part $I$) was originally proved in \cite{B-M-N}. We observed that it holds within conformal classes and also a similar statement (part $II$) holds by minor modifications on the original proof.

\begin{theorem}[Brendle-Marques-Neves \cite{B-M-N}]\label{thm:gluing}
Let $(M, g)$ be a Riemannian manifold with boundary $\partial M$. Suppose $\tilde g \in [g]$ is another metric on $M$ with the same induced metric on $\partial M$.

Then for any $\delta > 0$ and any neighborhood $K $ of $\partial M$ the following two statements hold.

$(I)$: If $H_g\geq H_{\tilde g}$, then there exists a metric $\hat g \in [ g]$ such that 
\begin{align*}
\begin{cases}
\hat g = g &\text{on $\partial M - K$}\\
\hat g = \tilde g &\text{in a neighborhood of $\partial M$}\\
R_{\hat g} \geq \min_{x \in M} \{R_{g} (x),  R_{\tilde g}(x)\} - \delta &\text{on $ M $}
\end{cases}
\end{align*}
and

$(II)$: If $H_g\leq H_{\tilde g}$, then there exists a metric $\hat g \in [ g]$ such that 
\begin{align*}
\begin{cases}
\hat g = g &\text{on $\partial M - K$}\\
\hat g = \tilde g &\text{in a neighborhood of $\partial M$}\\
R_{\hat g} \leq \max_{x \in M} \{R_{g} (x),  R_{\tilde g}(x)\} + \delta &\text{on $ M $.}
\end{cases}
\end{align*}\\

\end{theorem}
 
 Now we can prove the main theorem in this section.
 
\begin{proof}[Proof of Theorem \ref{thm:deformation}]
Applying Theorem \ref{thm:gluing}, for any $\delta > 0$, we can glue metrics $\tilde g_+$, $\tilde g_-$ from Proposition \ref{prop:boundary_metric} and $\hat g_t^+$, $\hat g_t^-$ from Proposition \ref{prop:perturb_metric} along $\partial\Omega_\varepsilon$ to get metrics $g_\delta^+$ and $g_\delta^-$  respectively, such that
$$R_{g_\delta^+} \geq \min_{x \in \Omega_\varepsilon} \{R_{\hat g_t^+} (x),  R_{\tilde g_+}(x)\} - \delta$$ 
and
$$R_{g_\delta^-} \leq \max_{x \in \Omega_\varepsilon} \{R_{\hat g_t^-} (x),  R_{\tilde g_-}(x)\} + \delta$$ 
on $\Omega_\varepsilon$. In particular, we get 
$$R_{g_\delta^+} > R_{\bar g}$$
and 
$$R_{g_\delta^- }< R_{\bar g}$$
inside $\Omega_\varepsilon$, if we choose $\delta$ sufficiently small.\\

Now we take 
\begin{align*}
g_+=
\begin{cases}
g_\delta^+ &\text{on $\Omega_\varepsilon$}\\
\tilde g_+ &\text{on $M - \Omega_\varepsilon$}
\end{cases}
\end{align*}
and
\begin{align*}
g_-=
\begin{cases}
g_\delta^- &\text{on $\Omega_\varepsilon$}\\
\tilde g_- &\text{on $M - \Omega_\varepsilon$.}
\end{cases}
\end{align*}

Clearly, $g_+$ and $g_-$ are smooth metrics which satisfy all requirement in the statement of Theorem \ref{thm:deformation}.

\end{proof}

Corollary \ref{cor:non_rig_closed} holds automatically, if we can justify the existence of a compact domain $\Omega$ with 
$$
\Lambda_1(\Omega, \bar g) < 0.
$$
 In fact, for closed manifolds this can be achieved with the aid of the following lemma:

\begin{lemma}\label{lem:reg_domain}
Let $(M^n, \bar g)$ be a closed Riemannian manifold with scalar curvature $R_{\bar g} > 0$. Then for any $\varepsilon > 0$, there exists a smooth domain $\Omega$ such that $\overline{\Omega} \neq M$ and its first Dirichlet eigenvalue of Laplacian satisfies $$\lambda_1(\Omega, \bar g) < \varepsilon.$$ In particular, we can take $\Omega$ such that $$\lambda_1(\Omega, \bar g) < \Lambda,$$ where $\Lambda : = \min_{x \in M} \frac{R_{\bar g}(x)}{n-1} > 0$. 
\end{lemma}

\begin{proof}

For any domain $\Omega \varsubsetneqq M$, we have
\begin{align*}
\lambda_1 (\Omega, \bar g) = \inf\{\frac{\int_\Omega |\nabla \varphi|^2 dv_{\bar g}}{\int_\Omega \varphi^2 dv_{\bar g} }:  \varphi|_{\partial \Omega} \equiv 0\}.
\end{align*}

For any $p \in M$ and any $r > 0$, let $B_r(p)$ and $B_{2r}(p)$ be geodesic balls around $p$ with radii $r$ and $2r$ respectively. \\

We take $\Omega_r : = M - B_r(p)$ and $0 \leq \varphi \leq 1$ a smooth test function which satisfies that 
\begin{align*}
\begin{cases}
\varphi = 1 &\text{on $M - B_{2r}(p)$}\\
\varphi = 0 & \text{on $B_r(p)$}\\
|\nabla \varphi| \leq \frac{2}{r} &\text{on $B_{2r} (p) - B_r(p)$}.
\end{cases}
\end{align*}

Then $\varphi$ is supported in $\Omega_r$ and we have $$\frac{\int_{\Omega_r} |\nabla \varphi|^2 dv_{\bar g}}{\int_{\Omega_r} \varphi^2 dv_{\bar g} } \leq \frac{4}{r^2} \cdot \frac{ Vol_{\bar g}(B_{2r}(p))}{ Vol_{\bar g} (M - B_{2r}(p)) } \leq \frac{ c_0 r^{n-2}}{ Vol_{\bar g} (M - B_{2r}(p)) }$$ for $r$ sufficiently small, where $c_0$ is a constant depends only on $n$.

Thus for $n \geq 3$, $$\lambda_1(\Omega_r, \bar g) \leq \frac{\int_{\Omega_r} |\nabla \varphi|^2 dv_{\bar g}}{\int_{\Omega_r} \varphi^2 dv_{\bar g} } \leq \frac{ c_0 r^{n-2}}{ Vol_{\bar g} (M - B_{2r}(p)) } \rightarrow 0,$$ as $r \rightarrow 0$ and hence we can find some $r_0 > 0$ such that for any $\varepsilon > 0$ we have $$\lambda_1(\Omega_{r_0}, \bar g) < \varepsilon.$$

As for $n=2$, if $M$ is orientable, then $M$ is diffeomorphic to the standard sphere $S^2$ since $R_{\bar g} > 0$. And there is a smooth function $u$ such that $e^{2u}\bar g = g_{S^2}$, which is the canonical spherical metric on $S^2$.

Let $q\in S^2$ be the south pole and $B_r(q)$ a geodesic ball with respect to $g_{S^2}$ centered at $q$ with radius $0< r < \pi$. Take $\Omega_r : = S^2 - B_r(q)$, then for any smooth function $\varphi \in C_0^{\infty} (\Omega_r)$, $$\lambda_1(\Omega_r, \bar g) \leq \frac{\int_{\Omega_r} |\nabla \varphi|^2 dvol_{\bar g}}{\int_{\Omega_r} \varphi^2 dvol_{\bar g}} = \frac{\int_{\Omega_r} |\nabla \varphi|^2 dvol_{g_{S^2}}}{\int_{\Omega_r} \varphi^2 e^{2u} dvol_{g_{S^2}}} \leq c_1\cdot \frac{\int_{\Omega_r} |\nabla \varphi|^2 dvol_{g_{S^2}}}{\int_{\Omega_r} \varphi^2  dvol_{g_{S^2}}},$$ where $c_1 := e^{-2 \min_M u}$ is independent of $r$. Therefore, $$\lambda_1(\Omega_r, \bar g) \leq c_1 \lambda_1(\Omega_r, g_{S^2}) \rightarrow 0,$$ as $r \rightarrow 0$ (cf. Theorem 6, P. 50 in \cite{Chavel}). Hence for any $\varepsilon > 0$, we can find an $r_0 > 0$ such that $$\lambda_1(\Omega_{r_0}, \bar g) < \varepsilon.$$

Suppose $M$ is not orientable, then $M$ is diffeomorphic to $\mathbb{R}P^2$, the real projective plane whose double covering $S^2$. Let $p$, $q$ be the north and south pole of $S^2$ and consider the domain $\tilde\Omega_r = S^2 - \left( B_r(p) \cup B_r(q)\right)$, $0 < r < \frac{\pi}{2}$ on $S^2$. For any $\varepsilon > 0$, take $r_0 > 0$ sufficiently small such that its quotient $\Omega_{r_0} \subset \mathbb{R}P^2$ is a domain. With similar calculations, we can find an $r_0$ such that $$\lambda_1(\Omega_{r_0}, \bar g) \leq c_1 \lambda_1 (\tilde\Omega_{r_0}, g_{S^2}) < \varepsilon.$$

\end{proof}

\begin{proof}[Proof of Corollary \ref{cor:non_rig_closed}]
Let $\Omega \subset M$ be the domain in Lemma \ref{lem:reg_domain} with $$\lambda_1(\Omega, \bar g) < \Lambda,$$ where $\Lambda : = \min_{x \in M} \frac{R_{\bar g}(x)}{n-1} > 0$. Then we have $$\Lambda_1(\Omega, \bar g) \leq \lambda_1 (\Omega, \bar g) - \Lambda < 0.$$
The conclusion follows from Theorem \ref{thm:deformation}.
\end{proof}

And for noncompact case:

\begin{proof}[Proof of corollary \ref{cor:non_rig_noncpt}]
 Since $(M, \bar g)$ has quadratic volume growth, we have the first Dirichlet eigenvalue of Laplacian satisfies that $$\lambda_1(M, \bar g) = 0$$ (see Proposition 9 in \cite{C-Y}) and it implies that for any $\varepsilon > 0$ there is a compactly contained domain $\Omega \subset M$ with $$\lambda_1 (\Omega, \bar g) < \varepsilon.$$ In particular, if we take $\varepsilon < \frac{Q}{2}$, we have $$\Lambda_1(\Omega, \bar g) \leq \lambda_1(\Omega, \bar g) - Q < - \frac{Q}{2} < 0.$$ Now by Theorem \ref{thm:deformation}, the corollary follows.

\end{proof}

\bibliographystyle{amsplain}

\end{document}